\documentclass
{amsart}
\usepackage{graphicx}
\newtheorem{thm}{Theorem}[section]
\theoremstyle{definition}
\newtheorem{cor}[thm]{Corollary}
\newtheorem{prop}[thm]{Proposition}
\newtheorem{defn}[thm]{Definition}
\newtheorem{lem}[thm]{Lemma}

\newtheorem{note}[thm]{Notation}
\newtheorem{rem}[thm]{Remark}
\newtheorem{ex}[thm]{Example}

\numberwithin{equation}{section}
\begin{document}
\title[Some generalizations of strongly prime ideals]
{Some generalizations of strongly prime ideals}
\author[H. Ansari-Toroghy,  F. Farshadifar, and S. Maleki-Roudposhti]%
{H. Ansari-Toroghy*, F. Farshadifar**, and S. Maleki-Roudposhti***}

\newcommand{\acr}{\newline\indent}

\address{\llap{*\,}Department of pure Mathematics\\
Faculty of mathematical
Sciences\\
University of Guilan\\
P. O. Box 41335-19141, Rasht, Iran}
\email{ansari@guilan.ac.ir}

\address{\llap{**\,}(Corresponding Author) Assistant Professor, Department of Mathematics,  Farhangian University, Tehran, Iran}
\email{f.farshadifar@cfu.ac.ir}

\address{\llap{***\,}Department of pure Mathematics\\
Faculty of mathematical
Sciences\\
University of Guilan\\
P. O. Box 41335-19141, Rasht, Iran}
\email{Sepidehmaleki.r@gmail.com}

\begin{abstract}
In this paper, we introduce the concepts of strongly 2-absorbing primary ideals (resp., submodules) and strongly 2-absorbing ideals (resp., submodules) as generalizations of strongly prime ideals. Furthermore, we investigate some basic properties of these
classes of ideals.
\end{abstract}
\maketitle
\section{Introduction}
\noindent
Throughout this paper, $R$ will denote an integral domain with quotient field $K$.
 Further, $\Bbb Z$, $\Bbb Q$,  and $\Bbb N$ will denote respectively the ring of integers, the field of rational numbers, and the set of natural numbers.

A proper ideal $I$ of $R$ is said to be \textit{strongly prime} if, whenever $xy \in I$ for elements $x, y \in K$, then $x \in I$ or $y \in I$ \cite{HH78}.
A proper ideal $I$ of $R$ is said to be \textit{strongly primary} if, whenever $xy \in I$ for elements $x, y \in K$, then $x \in I$ or $y^n \in I$  for some $n \geq 1$ \cite{BH02}.

The concept of $2$-absorbing ideals was introduced in \cite{Ba07}. A proper ideal $I$ of $R$ is said to be a \emph{2-absorbing ideal} of $R$ if whenever $a, b, c \in R$ and $abc \in I$, then $ab \in I$ or
 $ac \in I$ or $bc \in I$.
 In \cite{BUY14}, Badawi, et al. introduced the concept of 2-absorbing primary ideal which is a generalization of primary ideal. A proper ideal $I$ of $R$ is called a \emph{2-absorbing primary ideal} of $R$ if whenever $a, b, c \in R$ and $abc \in I$, then $ab\in I$ or $ac\in \sqrt{I}$ or  $bc\in \sqrt{I}$.

The purpose of this paper is to introduce the concepts of strongly 2-absorbing primary ideals (resp., submodules) and strongly 2-absorbing ideals (resp., submodules) as generalizations of strongly prime ideals. Furthermore, we investigate basic properties  of these
classes of ideals. 

Let $R$ be an integral domain with quotient field $K$. An ideal $I$ of $R$ is said to be a \emph{strongly 2-absorbing  primary ideal} if, whenever $xyz \in I$ for elements $x, y, z \in K$, we have either $xy \in I$ or  $(yz)^n \in I$ or $(xz)^m \in I$ for some $n,m\in \Bbb N $ (Definition \ref{d21.5}). A 2-absorbing ideal $I$ of $R$ is said to be a \emph{strongly 2-absorbing ideal} if, whenever $xyz \in I$ for elements $x, y, z \in K$, we have either  $xy \in I$ or $yz \in I$ or $xz \in I$ (Definition \ref{d1.5}). 
 Moreover, a submodule $N$ of an $R$-module $M$ is said to be \emph{strongly 2-absorbing primary} (resp.,  \emph{strongly 2-absorbing}) if $(N:_RM)$  is a  strongly 2-absorbing primary (resp.,  strongly 2-absorbing) ideal of $R$ (Definition \ref{d17.6} and \ref{d1.6}).

Let $R$ be an integral domain with quotient field $K$. 
In Section 2 of this paper, among other results, we prove that if $I$ is a strongly primary ideal of $R$, then $I$ is a strongly 2-absorbing  primary ideal of $R$ (Proposition \ref{p171.2}). Example \ref{c1171.2}, shows that the converse of Proposition \ref{p171.2} is not true in general. In Theorem \ref{t21.18}, we provide a useful characterization for strongly 2-absorbing primary ideals of $R$, where $R$ is a rooty domain.  In Theorem \ref{t771.7}, we show that for a strongly 2-absorbing primary ideal $I$  of $R$:
\begin{itemize}
  \item [(a)] If $J$ and $H$ are radical ideals of $R$, then $JH \subseteq I$ or $I^2 \subseteq J \cup H$;
  \item [(b)] If $J$ and $I$ are prime ideals of $R$, then $J$ and $I$ are comparable.
\end{itemize}
Furthermore, it is shown that if $P$ and $Q$ are non-zero strongly primary ideals of $R$, then $P \cap Q$ is a  strongly 2-absorbing primary ideal of $R$ (Theorem \ref{p17.22}).

In Section 3 of this paper, among other results, we prove that if $I$ is a strongly prime ideal of $R$, then $I$ is a strongly 2-absorbing ideal of $R$ (Proposition \ref{p11.2}). But the converse of Proposition \ref{p11.2} is not true in general (Example \ref{e11.2}, Proposition \ref{c11.2}, and Example \ref{ee1.3}). In Theorem \ref{t1.18},
we provide a useful characterization for a strongly 2-absorbing ideal of $R$. Also, we see that if $P$ and $Q$ are non-zero strongly prime ideals of $R$, then $P \cap Q$ is a  strongly 2-absorbing ideal of $R$ (Theorem \ref{p1.22}). Finally, it is proved that if $M$ is a Noetherian $R$-module, then $M$ contains a finite
number of minimal strongly 2-absorbing submodules (Theorem \ref{c63.10}).
\section{Strongly 2-absorbing primary ideals and submodules}
\begin{defn}\label{d21.5}
Let $R$ be an integral domain with  quotient field $K$.
We say that an ideal $I$ of $R$ is a \emph{strongly 2-absorbing  primary ideal} if, whenever $xyz \in I$ for elements $x, y, z \in K$, we have either  $xy \in I$ or $(yz)^n \in I$ or $(xz)^m \in I$ for some $n,m\in \Bbb N $.
\end{defn}

\begin{prop}\label{p171.2}
Let $R$ be an integral domain with quotient field $K$ and let $I$ be a strongly primary ideal of $R$. Then $I$ is a strongly 2-absorbing primary ideal of $R$.
\end{prop}
\begin{proof}
Let $xyz \in I$ for some $x, y, z \in K$. Then by assumption, either $xy\in I$  or $z^n \in I$ for some $n \geq 1$.  If $xy\in I$, then we are done. If $z^n \in I$, then $(zx)^n(zy)^n=(zxy)^nz^n \in I$. Thus again by assumption, either $(zx)^{n} \in I$ or $(yz)^{ns} \in I$ for some $s \geq 1$ as desired.
\end{proof}

The following example shows that the converse of Proposition \ref{p171.2} is not true in general.
\begin{ex}\label{c1171.2}
Let $K$ be a field of characteristic 2, and put $I = (X^2)K[[X^2,X^3]]$, where
$K[[X^2,X^3]]$ is the ring of formal power series over the indeterminates $X^2$ and
$X^3$. By considering the elements $X^3$ and $1/X$ in the quotient field $K((X))$, it is clear
that $I$ is not strongly primary. Now, let $fgh\in I$, where $f, g, h \in K((X))$. Then there exist units $u, v, w$ of the
$DVR$ $K[[X]]$ and integers $\alpha$, $\beta$, $\gamma$
 for which $f = uX^{\alpha}$, $g = vX^{\beta}$, and $h = wX^{\gamma}$. Then $fgh \in I$ implies that $\alpha +\beta+ \gamma\geq 2$; hence, $(\beta+ \gamma)+(\alpha + \gamma)+(\alpha +\beta)\geq 4$. Now,
if one of $\beta+ \gamma$ or $\alpha + \gamma$ is at least one, then correspondingly either $(gh)^2 \in I$ or
$(fh)^2 \in I$. On the other hand, if both $\beta+ \gamma$
 and $\alpha + \gamma$ are at most 0, then $\alpha + \beta\geq 4$.
However, this would mean that $fg \in I$. Therefore, $I$ must be a strongly 2-absorbing
primary ideal of $K[[X^2, X^3]]$.
\end{ex}

\begin{note}\label{n171.2}
For a subset $S$ of $R$, we define $E(S)$ by
$$
E(S)=\{x \in K:\\\ x^n \not \in S \ for \ each \ n\geq 1 \}.
$$
\end{note}\label{c171.2}

Let $R$ be an integral domain with  quotient field $K$.
An ideal $I$ of $R$ is called \textit{strongly radical} if whenever $x \in K$ satisfies $x^n\in I$ for some $n \geq 1$,
then $x \in I$ \cite{AA89}.

Following \cite{SS90}, an integral domain $R$ is called \textit{rooty} if each radical
ideal of $R$ is strongly radical (equivalently, each prime ideal of $R$ is strongly radical.
Thus valuation domains are rooty domains \cite{AP99}).

\begin{thm}\label{t21.18}
Let $R$ be an integral domain with quotient field $K$ and let $I$ be an ideal of $R$.
Consider the following statements:
\begin{itemize}
\item [(a)]  $I$ is a 2-absorbing primary ideal of $R$ and for each $x, y \in K$ with $xy \not \in I$ we have $x^{-1}I \cap E(I)=\emptyset$ or $y^{-1}I \cap E(I)=\emptyset$.
\item [(b)]  $I$ is a strongly 2-absorbing primary ideal of $R$.
\end{itemize}
Then $(a) \Rightarrow (b)$. Moreover, if $K\setminus E(I)$ is
closed under addition (in particular, if $R$ is rooty), then $(b) \Rightarrow (a)$.
\end{thm}
\begin{proof}
$(a) \Rightarrow (b)$
Let $xyz \in I$ for some $x, y, z \in K$ and $xy \not \in I$. Then by part (a), either $x^{-1}I \cap E(I)=\emptyset$ or $y^{-1}I \cap E(I)=\emptyset$. If $x^{-1}I \cap E(I)=\emptyset$, then $yz=yzxx^{-1} =(yzx)x^{-1} \in x^{-1}I $ implies that $(yz)^n \in I$ for some $n \geq 1$. Similarly, if  $y^{-1}I \cap E(I)=\emptyset$, then we have $(xz)^m \in I$  for some $m \geq 1$, as needed.

$(b) \Rightarrow (a)$
Assume on the contrary that $x, y \in K$ with $xy \not \in I$  and $x^{-1}I \cap E(I)\not =\emptyset$ and  $y^{-1}I \cap E(I)\not=\emptyset$. Then there exist $a, b\in I$ such that $x^{-1}a \in E(I)$ and $y^{-1}b \in E(I)$.
Now as $I$ is a strongly 2-absorbing primary ideal of $R$, we have $(x)(y)(x^{-1}y^{-1}a)=a \in I$ implies that $(y^{-1}a )^n\in I$ for some $n \geq 1$. In a similar way we have $(x^{-1}b)^m \in I$ for some $m \geq 1$. On the other hand,
$$
a+b=(x)(y)(x^{-1}y^{-1}(a+b))\in I
$$
implies that either $xy \in I$ or  $(x^{-1}(a+b))^s \in I$ or $(y^{-1}(a+b))^t \in I$. Therefore, as  $K\setminus E(I)$ is closed under addition, either $xy \in I$ or
$x^{-1}a\not \in E(I)$ or $y^{-1}b\not \in E(I)$, which is a contradiction.
\end{proof}

\begin{thm}\label{t1.7}
Let $R$ be an integral domain with quotient field $K$ and $I$ be an ideal of $R$. Consider the following:
\begin{itemize}
  \item [(a)]  If $xyz \in I$ for elements $x, y, z \in K$, we have either  $xy \in I$ or $yz \in \sqrt{I}$ or $xz \in \sqrt{I}$.
  \item [(b)]  If $xyz \in I$ for elements $x, y, z \in K$, we have either  $xy \in I$ or $(yz)^n \in I$ or $(xz)^m \in I$ for some $n,m \geq 1$ (i.e.,  $I$ is a strongly 2-absorbing primary ideal of $R$).
\end{itemize}
Then $(a)\Rightarrow (b)$. Moreover, if $R$ is a rooty domain, then $(b)\Rightarrow (a)$.
\end{thm}
\begin{proof}
$(a)\Rightarrow (b)$ This is clear.

$(b)\Rightarrow (a)$
Let  $xyz \in I$ for elements $x, y, z \in K$. If $xy \not \in I$, then we have either  $(yz)^n \in I$ or $(xz)^m \in I$ for some $n,m \geq 1$ by part (b). Since $R$ is a rooty domain, $yz\in \sqrt{I}$ or $xz \in \sqrt{I}$, as needed.
\end{proof}

\begin{thm}\label{t771.7}
Let $R$ be an integral domain with  quotient field $K$ and let $I$ be a strongly 2-absorbing primary ideal of $R$. Then we have the following:
\begin{itemize}
  \item [(a)] If $J$ and $H$ are radical ideals of $R$, then $JH \subseteq I$ or $I^2 \subseteq J \cup H$.
  \item [(b)] If $J$ and $I$ are prime ideals of $R$, then $J$ and $I$ are comparable.
\end{itemize}
\end{thm}
\begin{proof}
(a) Suppose that $J$ and $H$ are radical ideals of $R$ such that $JH \not \subseteq I$. Then there exist $a \in J$ and $b \in H$ such that $ab \in JH \setminus I$. Let $x, y \in I$. Then $(xy/ab)(a/x)(b/1) \in I$ implies that either $(a/x)(b/1) \in I$ or $((xy/ab)(a/x))^n \in I$ or $((xy/ab)(b/1))^m \in I$ for some $n,m \geq 1$. Thus either $x(ab/x) \in xR \subseteq I$ or $(b(y/b))^n \in b^nI \subseteq b^nR \subseteq H$ or $(a(xy/a) )^m\in a^mI \subseteq a^mR \subseteq J$. Hence, either $ab \in I$ or $y^n \in H$ or $(xy)^m \in J$. Since $ab \not \in I$, we have either  $y \in \sqrt{H}=H$ or $xy \in \sqrt{J}=J$. Therefore, $xy \in J \cup H$. This implies that $I^2 \subseteq J \cup H$, as desired.

(b) The result follows from the fact that $J^2 \subseteq I$ or $I^2 \subseteq J$  by part (a).
\end{proof}

\begin{cor}\label{c1.8}
Let $R$ be an integral domain with quotient field $K$ and $Q$ be a maximal ideal of $R$. If  $Q$ is a strongly 2-absorbing primary ideal of $R$, then $R$ is a local ring with maximal ideal $Q$.
\end{cor}
\begin{proof}
It follows from Theorem \ref{t771.7}.
\end{proof}

\begin{thm} \label{p17.22}
Let $R$ be an integral domain with quotient field $K$ and let $P$ and $Q$ be nonzero strongly primary ideals of $R$.  Then $P \cap Q$ is a
 strongly 2-absorbing primary ideal of $R$.
\end{thm}
\begin{proof}
Suppose $(xy)z \in  P \cap Q$ and $x, y, z \in K$. Then $(xy)z \in P$ and $(xy)z  \in  Q$. Since $P$
is strongly primary, so either $xy\in P$ or $z^n\in P$ for some $n\geq 1$. If $xy\in P$, then either $x \in P$ or $y^m\in P$ for some $m \geq 1$. Similarly, $x \in Q$ or $y^t \in Q$ or $z^s \in Q$ for some $s,t \geq 1$. First assume that  $x \in P$ and $x \in Q$. Then $(xy)y^{-1}=x \in P$ implies that $xy \in P$ or $(y^{-1})^h\in P$ for some $h\geq 1$. Similarly, $xy \in Q$ or $(y^{-1})^g\in Q$ for some $g \geq 1$. If $(y^{-1})^g\in Q\subseteq R$ or $(y^{-1})^h\in P\subseteq R$, then $(xz)^h=(xyz)^h(y^{-1})^h \in P \cap Q$ or $(xz)^g=(xyz)^g(y^{-1})^g \in P \cap Q$ by definition of an ideal. Otherwise, $xy\in P\cap Q$ as requested.
If the statements
above lead to different elements in $P$ and $Q$, we still have that the intersection is
strongly 2-absorbing primary. For example, if $z^n\in P$ and $y^t \in Q$, then clearly $(xy)^{nt} \in P$ and $(xy)^{nt} \in Q$ by
definition of an ideal, thus $(xy)^{nt} \in  P \cap Q$.
\end{proof}

\begin{prop} \label{p71.23}
Let $R$ be an integral domain with quotient field $K$ and  $S$ be a multiplicatively closed subset of
$R$. If $I$ is a strongly 2-absorbing primary ideal of $R$ such that $S\cap I = \emptyset$, then  $S^{-1}I$ is a strongly 2-absorbing primary ideal of $S^{-1}R$.
\end{prop}
\begin{proof}
Assume that $a, b, c \in K$ such that $abc \in  S^{-1}I$. Then there exists $s \in S$ such that  $(sa)(b)c=sabc \in I$. Since  $I$ is  a strongly 2-absorbing primary ideal of $R$, this implies that
either $(sa)c\in I$ or $((b)c)^n=(bc)^n \in I$ or $((sa)(b))^m=(sab)^m\in I$ for some $n,m \geq 1$. Thus $ac=(sa)c/s\in s^{-1}I$ or $(bc)^n=((b)c/1)^n \in s^{-1}I$ or $(ab)^m=((sa)(b)/s)^m\in s^{-1}I$ as needed.
\end{proof}

\begin{defn}\label{d17.6}
Let $R$ be an integral domain with quotient field $K$ and $M$ be an $R$-module.
We say that a submodule $N$ of $M$ is a \emph{strongly 2-absorbing  primary} if, $(N:_RM)$  is a  strongly 2-absorbing  primary ideal of $R$.
\end{defn}

\begin{prop}\label{p672.20}
Let $R$ be an integral domain with quotient field $K$,  $M$ be an $R$-module, and $N_1$, $N_2$ be two submodules of $M$ with $(N_1:_RM)$ and $(N_2:_RM)$  strongly primary ideals of $R$. Then $N_1\cap N_2$ is a strongly 2-absorbing primary submodule of $M$.
\end{prop}
\begin{proof}
Since $(N_1\cap N_2:_RM)=(N_1:_RM) \cap (N_2:_RM)$, the result follows from Proposition \ref{p17.22}.
\end{proof}

\begin{prop}
Let $R$ be an integral domain with quotient field $K$, $N$ be submodule of a finitely generated  $R$-module $M$, and let $S$ be a multiplicatively closed subset of $R$. If $N$ is a  strongly 2-absorbing primary submodule and  $(N:_RM) \cap S=\emptyset$, then $S^{-1}N$ is a  strongly 2-absorbing primary $S^{-1}R$-submodule of $S^{-1}M$.
\end{prop}
\begin{proof}
As $M$ is finitely generated, $(S^{-1}N:_{S^{-1}R}S^{-1}M)=S^{-1}(N:_RM)$ by \cite[Lemma 9.12]{SH90}. Now the result follows from Proposition \ref{p71.23}.
\end{proof}

\begin{prop}\label{p17.12}
Let $R$ be an integral domain with quotient field $K$ and $M$ be an $R$-module.
Let $N$ be a strongly 2-absorbing primary submodule of $M$. Then we have the following.
\begin{itemize}
  \item [(a)] If $r \in K$ such that $r^{-1}\in R$, then $(N:_Mr)$ is a strongly 2-absorbing primary submodule of $M$.
  \item [(b)] If $f : M \rightarrow \acute{M}$ is a monomorphism of $R$-modules, then $N$ is a strongly 2-absorbing primary submodule of $M$ if and only if $f(N)$ is a strongly 2-absorbing primary submodule of $f(M)$.
\end{itemize}
\end{prop}
\begin{proof}
(a) Let $xyz \in ((N:_Mr):_RM)$ for some $x, y, z \in K$. Then  $rxyz \in (N:_RM)$. Thus as $N$ is a strongly 2-absorbing primary
 submodule, either $rxy \in (N:_RM)$ or $(rxz)^n \in (N:_RM)$ or $(yz)^m \in (N:_RM)$ for some $n, m \geq 1$.  Hence either $xy=r^{-1}rxy \in r^{-1}(N:_RM)\subseteq (N:_RM)$ or $(xz)^n=(r^{-1}rxz)^n \in r^{-1}(N:_RM)\subseteq (N:_RM)$ or $(yz)^m \in (N:_RM)$, as needed.

(b) This follows from the fact that $(N:_RM)=(f(N):_Rf(M))$.
\end{proof}

\section{Strongly 2-absorbing ideals and submodules}
\begin{defn}\label{d1.5}
Let $R$ be an integral domain with  quotient field $K$.
We say that a 2-absorbing ideal $I$ of $R$ is a \emph{strongly 2-absorbing ideal} if, whenever $xyz \in I$ for elements $x, y, z \in K$, we have either  $xy \in I$ or $yz \in I$ or $xz \in I$.
\end{defn}

\begin{prop}\label{p11.2}
Let $R$ be an integral domain with quotient field $K$ and let $I$ be a strongly prime ideal of $R$. Then $I$ is a strongly 2-absorbing ideal of $R$.
\end{prop}
\begin{proof}
Let $xyz \in I$ for some $x, y, z \in K$. Then by assumption, either $xy\in I$  or $z \in I$. If $xy\in I$, then we are done. If $z \in I$, then $zxyz \in I$. Thus again by assumption, either $zx \in I$ or $yz \in I$ as desired.
\end{proof}

The following theorem is a characterization for a strongly 2-absorbing ideal of $R$.
\begin{thm}\label{t1.18}
Let $R$ be an integral domain with  quotient field $K$ and let $I$ be a 2-absorbing ideal of $R$. Then the following statements are  equivalent:
\begin{itemize}
  \item [(a)] $I$ is a strongly 2-absorbing ideal of $R$;
  \item [(b)] For each $x, y \in K$ with $xy \not \in R$ we have either $x^{-1}I \subseteq I$ or $y^{-1}I \subseteq I$.
\end{itemize}
\end{thm}
\begin{proof}
$(a) \Rightarrow (b)$
Assume on the contrary that $x, y \in K$ with $xy \not \in R$  and neither $x^{-1}I \not\subseteq I$ nor $y^{-1}I \not\subseteq I$. Then there exist $a , b \in I$ such that $x^{-1}a \not \in I$ and $y^{-1}b \not \in I$.
Now as $I$ is a strongly 2-absorbing ideal of $R$, we have $(x)(y)(x^{-1}y^{-1}a)=a \in I$ implies that $y^{-1}a \in I$. In the similar way we have $x^{-1}b \in I$. On the other hand,
$$
a+b=(x)(y)(x^{-1}y^{-1}(a+b))\in I
$$
implies that either $xy \in I$ or  $x^{-1}(a+b) \in I$ or $y^{-1}(a+b) \in I$. Therefore, either $xy \in R$ or
$x^{-1}a \in I$ or $y^{-1}b \in I$, a contradiction.

$(b) \Rightarrow (a)$
Let $xyz \in I$ for some $x, y, z \in K$. If $xy \in R$, $xz  \in R$, and $yz \in R$, then we are done since $I$ is a 2-absorbing ideal of $R$. So suppose without loss of generality that  $xy \not \in R$. Then by part (b), either $x^{-1}I \subseteq I$ or
$y^{-1}I \subseteq I$. If  $x^{-1}I \subseteq I$, then $yz=yzxx^{-1} =(yzx)x^{-1} \in x^{-1}I \subseteq I$. Similarly, if
$y^{-1}I \subseteq I$, then we have $xz \in I$, as desired.
\end{proof}

\begin{cor}
Let $R$ be an integral domain with  quotient field $K$ and let $I$ be a strongly 2-absorbing ideal of $R$. Then for each $x, y \in K$ with $xy \not \in R$ we have either $I \subseteq Rx$ or $I \subseteq Ry$.
\end{cor}
\begin{proof}
Let $x, y \in K$ with $xy \not \in R$. Then by Theorem \ref{t1.18} $(a)\Rightarrow (b)$,  we have either $x^{-1}I \subseteq I$ or $y^{-1}I \subseteq I$. Thus either $I\subseteq Ix \subseteq Rx$ or $I \subseteq Ix\subseteq Ry$.
\end{proof}

Example \ref{e11.2}, Proposition \ref{c11.2}, and Example \ref{ee1.3} show that the converse of Proposition \ref{p171.2} is not true in general. 

\begin{ex}\label{e11.2}
If $Q$ is the maximal ideal of a non-trivial $DVR$,
$V$, then $Q^2$ is a strongly 2-absorbing ideal of $V$ that is not a strongly prime ideal,
since $Q^2$ is not even a prime ideal of $V$.
\end{ex}

\begin{prop}\label{c11.2}
Let $R$ be an integral domain with a prime ideal $P$ such that there
exists a discrete valuation overring $(V, Q)$ of $R$ centered at $P$ (that is, $Q\cap R = P$),
where $Q = xV$. Suppose that $ux^k \in P$ for all units $u$ of $V$ and natural numbers
$k\geq 2$, but there is no unit $u$ of $V$ for which $ux \in P$. Then $P$ is a strongly 2-absorbing
ideal of $R$ that is not a strongly prime ideal.
\end{prop}
\begin{proof}
The fact that $P$ is not a strongly prime ideal of $R$ is immediate from the
fact that $x^2 \in P$, but $x \not \in P$, by assumption.
Now, since $P$ is a prime ideal of $R$, it is necessarily a 2-absorbing ideal of $R$. Let
$y$ and $z$ be elements of the quotient field of $R$ for which $yz \not \in R$. By Theorem \ref{t1.18},
it suffices to show that either $y^{-1}P\subseteq P$ or $z^{-1}P \subseteq P$. Observe that there exist
units $u$ and $v$ of $V$ and integers $\alpha$ and $\beta$ for which $y = ux^{\alpha}$ and $z = vx^{\beta}$. Since
$yz \not \in R$, it must be the case that $\alpha+\beta \leq 1$. However, this means that either $\alpha \leq 0$
or $\beta \leq 0$. As such, either $-\alpha+\gamma \geq 2$ or $-\beta+\gamma \geq 2$ for all integers
$\gamma \geq 2$, from which it follows that either $y^{-1}P\subseteq P$ or $z^{-1}P \subseteq P$ as needed.
\end{proof}

\begin{ex}\label{ee1.3}
If $K$ is a field, then the ideal $(X^2,X^3)$ in $K[[X^2,X^3]]$ the ring of formal power series in the indeterminates
$X^2$ and $X^3$ over $K$ is an example of a strongly 2-absorbing prime ideal that is not strongly prime.
\end{ex}

\begin{prop}\label{p1.3}
Let $R$ be an integral domain with  quotient field $K$,  $I$ be a strongly 2-absorbing ideal of $R$, and $Q$ be a prime ideal of $R$ which is properly contained in $I$. Then $I/Q$ is a strongly 2-absorbing ideal of $R/Q$.
\end{prop}
\begin{proof}
Clearly, $I/Q$ is a 2-absorbing ideal of $R/Q$. Now let $\phi :R\rightarrow R/Q$ denote the canonical homomorphism. Suppose that
$x_1=\phi(y_1)/\phi(z_1)$ and $x_2=\phi(y_2)/\phi(z_2)$ are elements of the quotient field of $R/Q$ such that $x_1x_2 \not \in R/Q$. Then $(y_1/z_1)(y_2/z_2) \not \in R$. Hence if $a \in I$, we have $(z_1/y_1)a \in I$ or $(z_2/y_2)a \in I$ by using Theorem \ref{t1.18}. We can assume without loss of generality that $(z_1/y_1)a \in I$. It follows that $(\phi (z_1)/\phi (y_1))\phi(a) \in I/Q$. Thus $x^{-1}(I/Q)\subseteq I/Q$, as needed.
\end{proof}

\begin{rem}\label{d11.5}
Clearly, every strongly 2-absorbing ideal of $R$ is a 2-absorbing ideal of $R$. But the converse is not true in general. Because for
example,  if we consider the integral domain $\Bbb Z$, then $K=\Bbb Q$ and $(8/15)(3/2)(5/2) = 2 \in 2\Bbb Z$ implies that
$2\Bbb Z$ is not a strongly 2-absorbing ideal of $\Bbb Z$. But $2\Bbb Z$ is a 2-absorbing ideal of $\Bbb Z$.
\end{rem}

\begin{defn}\label{d1.16}
We say that an integral domain $R$ is a \emph{2-absorbing pseudo-valuation domain} if every 2-absorbing ideal of $R$ is a strongly 2-absorbing ideal of $R$.
\end{defn}

\begin{prop}\label{c11.17}
Every valuation domain is a 2-absorbing pseudo-valuation
domain.
\end{prop}
\begin{proof}
Let $V$ be a valuation domain, and let $I$ be a 2-absorbing ideal of $V$. Suppose $xyz \in I$, where $x , y, z \in K$, the quotient field of $V$. If  $x$, $y$, and $z$ are in $V$, we are done. Suppose without loss of generality that $x \not \in  V$. Since $V$ is a
valuation domain, we have $x^{-1} \in V$. Hence $yz = (x^{-1})(xyz) \in I$, as needed.
\end{proof}

\begin{defn}\label{d1.19}
Let $R$ be an integral domain with quotient field $K$. We say that a non-zero prime ideal $P$ of $R$ is a \emph{strongly semiprime} if whenever $x^{2} \in P $ for element $x \in K$, we have $x \in P$.
\end{defn}

\begin{rem}\label{r1.13}
Let $R$ be an integral domain with quotient field $K$. Clearly every non-zero strongly prime ideal of $R$ is a strongly semiprime ideal of $R$. But as we see in the following example the converse is not true in general.
\end{rem}

\begin{ex}\label{e1.13}
Consider an integral domain $\Bbb Z$. Then $K=\Bbb Q$ and $(4/3)(3/2) =2 \in 2\Bbb Z$ implies that
$2\Bbb Z$ is not a  strongly prime ideal of $\Bbb Z$. But $2\Bbb Z$ is  a  strongly semiprime ideal of $\Bbb Z$.
\end{ex}

\begin{prop}\label{p1.20}
Let $R$ be an integral domain with quotient field $K$.
\begin{itemize}
\item [(a)]
If $P$ is a strongly semiprime and strongly 2-absorbing ideal of $R$, then $P$  is a strongly prime ideal of $R$.
\item [(b)]
If $P_1$ and $P_2$ are strongly semiprime ideals of $R$, then $P_1 \cap P_2$  is a  strongly semiprime ideal of $R$.
\end{itemize}
\end{prop}
\begin{proof}
(a) Let $P$ be a  strongly semiprime and 2-absorbing ideal of $R$ and let $x \in K \setminus R$. Then as $P$ is  strongly semiprime $x^{2} \not \in P$. Since $P$ is strongly 2-absorbing, this implies that $x^{-1}P \subseteq P$ by Theorem \ref{t1.18}.  Now the result follows from \cite[Proposition 1.2]{HH78}.

(b) This is clear.
\end{proof}

\begin{thm} \label{p1.22}
Let $R$ be an integral domain with quotient field $K$ and let $P$ and $Q$ be non-zero strongly prime ideals of $R$.  Then $P \cap Q$ is a
 strongly 2-absorbing ideal of $R$.
\end{thm}
\begin{proof}
The proof is similar to that of  Theorem \ref{p17.22}.
\end{proof}

\begin{prop} \label{p1.23}
Let $R$ be an integral domain with quotient field $K$ and let  $I$ be a strongly 2-absorbing ideal of $R$. Then we have the following:
\begin{itemize}
\item [(a)] $\sqrt{I}$ is a strongly 2-absorbing ideal of $R$ and $x^2\in I$ for every $x\in  \sqrt{I}$.
\item [(b)] If $S$ is a multiplicatively closed subset of $R$ such that $S\cap I = \emptyset$, then  $S^{-1}I$ is a strongly 2-absorbing ideal of $S^{-1}R$.
\end{itemize}
\end{prop}
\begin{proof}
(a) Since $I$ is a strongly 2-absorbing ideal of $R$, observe that $x^2 \in I$ for every  $x \in \sqrt{I}$.
Let $x, y, z \in  K$ such that $xyz \in \sqrt{I}$. Then $(xyz)^2 = x^2y^2z^2 \in I$. Since $I$ is a 2-absorbing ideal of $R$, we may assume without loss of generality that $x^2y^2 \in  I$.  Now since $(xy)^2 = x^2y^2 \in I$, we have $xy \in \sqrt{I}$ as desired.

(b) The proof is similar to that of  Proposition \label{}\ref{p71.23}.
\end{proof}

\begin{thm}\label{t71.7}
Let $R$ be an integral domain with quotient field $K$ and let  $I$ be a strongly 2-absorbing ideal of $R$. Then we have the following.
\begin{itemize}
  \item [(a)] If $J$ and $H$ are ideals of $R$, then $JH \subseteq I$ or $I^2 \subseteq J \cup H$.
  \item [(b)] If $J$ and $I$ are prime ideals of $R$, then $J$ and $I$ are comparable.
\end{itemize}
\end{thm}
\begin{proof}
The proof is similar to that of Theorem  \ref{t771.7}.
\end{proof}

\begin{cor}\label{c71.8}
Let $R$ be an integral domain with quotient field $K$ and $Q$ be a maximal ideal of $R$. If $Q$ is a strongly 2-absorbing ideal of $R$, then $R$ is a local ring with maximal ideal $Q$.
\end{cor}
\begin{proof}
This follows from Theorem \ref{t71.7} (b).
\end{proof}

Recall that if $K$ is the field of fractions of $R$, then an intermediate ring in the extension
$R\subseteq K$ is called an \emph{overring} of $R$.
\begin{prop}\label{p1.11}
Let $R$ be an integral domain with quotient field $K$,  $I$ be a strongly 2-absorbing ideal of $R$, and let $T$ be an overring of $R$. Then $IT$ is a strongly 2-absorbing ideal of $T$.
\end{prop}
\begin{proof}
Let $x, y \in K$ and $xy \not \in T$. Then $xy \not \in R$. Thus by Theorem \ref{t1.18}, either $x^{-1}I \subseteq I$ or $y^{-1}I \subseteq I$. Therefore, either $x^{-1}IT \subseteq IT$ or $y^{-1}I T\subseteq IT$. Hence $IT$ is a strongly 2-absorbing ideal of $T$, again by Theorem \ref{t1.18}.
\end{proof}

\begin{prop}\label{p1.9}
Let $R$ be an integral domain with quotient field $K$ and let $\{I_{\lambda}\}_{\lambda \in \Lambda}$ be a chain of  strongly 2-absorbing ideals of $R$. Then
$\sum_{\lambda \in \Lambda}I_{\lambda}$ is a strongly 2-absorbing ideal of $R$.
\end{prop}
\begin{proof}
Suppose that $x, y \in K$ with $xy \not \in R$ and we have $x^{-1}\sum_{\lambda \in \Lambda}I_{\lambda} \not\subseteq \sum_{\lambda \in \Lambda}I_{\lambda}$ and $y^{-1}\sum_{\lambda \in \Lambda}I_{\lambda} \not\subseteq \sum_{\lambda \in \Lambda}I_{\lambda}$. Then there exist $\alpha , \beta \in \Lambda$ such that $x^{-1}I_{\alpha}\not \subseteq \sum_{\lambda \in \Lambda}I_{\lambda}$ and $y^{-1}I_{\beta} \not\subseteq \sum_{\lambda \in \Lambda}I_{\lambda}$. Hence,
$x^{-1}I_{\alpha}\not \subseteq I_{\alpha}$ and $y^{-1}I_{\beta} \not\subseteq I_{\beta}$.
Thus  $y^{-1}I_{\alpha} \subseteq I_{\alpha}$ and $x^{-1}I_{\beta} \subseteq  I_{\beta}$. By assumption, $I_{\alpha}\subseteq I_{\beta}$ or $I_{\beta}\subseteq I_{\alpha}$. This implies that $x^{-1}I_{\alpha}\subseteq x^{-1}I_{\beta}\subseteq I_{\beta}\subseteq \sum_{\lambda \in \Lambda}I_{\lambda}$ or $y^{-1}I_{\beta}\subseteq y^{-1}I_{\alpha}\subseteq I_{\alpha}\subseteq \sum_{\lambda \in \Lambda}I_{\lambda}$. This is a contradiction. Thus by Theorem \ref{t1.18}, $\sum_{\lambda \in \Lambda}I_{\lambda}$ is a strongly 2-absorbing ideal of $R$.
\end{proof}

Recall that a \emph{chained ring} is any ring whose set of ideals is totally ordered by inclusion.
\begin{cor}\label{c1.10}
If $R$ is a chained ring and contains a strongly 2-absorbing ideal, then $R$ contains a unique largest strongly 2-absorbing ideal.
\end{cor}
\begin{proof}
This is proved easily by using Zorn's Lemma and Proposition \ref{p1.9}.
\end{proof}

\begin{defn}\label{d1.6}
Let $R$ be an integral domain with quotient field $K$ and $M$ be an $R$-module.
We say that a submodule $N$ of $M$ is a \emph{strongly 2-absorbing} if $(N:_RM)$  is a  strongly 2-absorbing ideal of $R$.
\end{defn}
An $R$-module $M$ is said to be a \emph{multiplication module} if for every submodule $N$ of $M$ there exists an ideal $I$ of $R$ such that $N=IM$ \cite{Ba81}.
\begin{prop}\label{p1.10}
Let $R$ be an integral domain which is a chained ring with quotient field $K$ and $M$ be a faithful finitely generated multiplication $R$-module.
If $\{N_i\}_{i \in I}$ is a family of strongly 2-absorbing submodules of $M$, then $\sum_{i \in I}N_i$ is a strongly 2-absorbing submodule of $M$.
\end{prop}
\begin{proof}
This follows from Proposition \ref{p1.9} and the fact that
$$
(\sum_{i \in I}(N_i:_RM)M:_RM)=\sum_{i \in I}(N_i:_RM)
$$
 by \cite[Theorem 3.1]{BS89}.
\end{proof}

\begin{prop}\label{p1.12}
Let $R$ be an integral domain with quotient field $K$ and $M$ be an $R$-module.
Then we have the following:
\begin{itemize}
  \item [(a)] If $N$ is a strongly 2-absorbing submodule of $M$ and $r \in K$ such that $r^{-1}\in R$, then $(N:_Mr)$ is a strongly 2-absorbing submodule of $M$.
  \item [(b)] If $f : M \rightarrow \acute{M}$ is a monomorphism of $R$-modules, then $N$ is a strongly 2-absorbing submodule of $M$ if and only if $f(N)$ is a strongly 2-absorbing submodule of $f(M)$.
  \item [(c)] If $N_1$, $N_2$ are two submodules of $M$ with $(N_1:_RM)$ and $(N_2:_RM)$  strongly prime ideals of $R$, then $N_1\cap N_2$ is a strongly 2-absorbing submodule of $M$.
\end{itemize}
\end{prop}
\begin{proof}
(a) The proof is similar to that of  Proposition \ref{p17.12} (a).

(b) The proof is similar to that of  Proposition \ref{p17.12} (b).

(c) Since $(N_1\cap N_2:_RM)=(N_1:_RM) \cap (N_2:_RM)$, the result follows from Proposition \ref{p1.22}.
\end{proof}

\begin{prop}
Let $R$ be an integral domain with quotient field $K$, $N$ be a submodule of a finitely generated $R$-module $M$, and let $S$ be a multiplicatively closed subset of $R$. If $N$ is a  strongly 2-absorbing submodule and  $(N:_RM) \cap S=\emptyset$, then $S^{-1}N$ is a  strongly 2-absorbing $S^{-1}R$-submodule of $S^{-1}M$.
\end{prop}
\begin{proof}
As $M$ is finitely generated, $(S^{-1}N:_{S^{-1}R}S^{-1}M)=S^{-1}(N:_RM)$ by \cite[Lemma 9.12]{SH90}. .Now the result follows from Proposition \ref{p1.23}.
\end{proof}

\begin{prop}\label{p67.21}
Let $R$ be an integral domain with quotient field $K$, $M$ be an $R$-module, and let $\{K_i\}_{i \in I}$ be a chain of strongly
2-absorbing submodules of $M$. Then $\cap_{i \in I}K_i$ is a strongly 2-absorbing submodule of $M$.
\end{prop}
\begin{proof}
Let $a, b, c \in K$ and $abc \in (\cap_{i \in I}K_i:_RM)= \cap_{i \in I}(K_i:_RM)$. Assume to the contrary that $ab \not \in \cap_{i \in I}(K_i:_RM)$, $bc\not \in \cap_{i \in I}(K_i:_RM)$, and $ac\not \in \cap_{i \in I}(K_i:_RM)$. Then there are $m,n, t \in I$ where $ab \not \in (K_n:_RM)$, $bc \not \in (K_m:_RM)$, and $ac \not \in (K_t:_RM)$. Since $\{K_i\}_{i \in I}$ is a chain, we can assume without loss of generality that $K_m \subseteq K_n \subseteq K_t$. Then $$
(K_m:_RM) \subseteq (K_n:_RM) \subseteq (K_t:_RM).
$$
As $abc \in (K_m:_RM)$, we have  $ab \in (K_m:_RM)$ or $ac \in (K_m:_RM)$ or $bc \in (K_m:_RM)$. In any case, we have a contradiction.
\end{proof}

\begin{defn}\label{d67.22}
Let $R$ be an integral domain with quotient field $K$.
We say that a strongly 2-absorbing submodule $N$ of an $R$-module $M$
is a \emph {minimal strongly 2-absorbing submodule} of a submodule
$H$ of $M$, if $H \subseteq N$ and there does not exist a strongly 2-absorbing submodule $T$ of $M$ such that $H \subset T \subset N$.
\end{defn}
It should be noted that a minimal strongly 2-absorbing submodule of $M$ means that a minimal strongly 2-absorbing submodule of the submodule $0$ of $M$.

\begin{lem}\label{l67.23}
Let $R$ be an integral domain with quotient field $K$ and let $M$ be an $R$-module. Then every strongly 2-absorbing submodule of $M$ contains a minimal strongly 2-absorbing submodule of $M$.
\end{lem}
\begin{proof}
This is proved easily by using Zorn's Lemma and Proposition \ref{p67.21}.
\end{proof}

\begin{thm}\label{c63.10}
Let $R$ be an integral domain with quotient field $K$ and let $M$ be a Noetherian $R$-module. Then $M$ contains a finite
number of minimal strongly 2-absorbing submodules.
\end{thm}
\begin{proof}
Suppose that the result is false. Let $\Sigma$ denote the collection of all proper submodules
$N$ of $M$ such that the module $M/N$ has an infinite number of minimal
strongly 2-absorbing submodules. Since $0 \in \Sigma$, we have $\Sigma \not = \emptyset$. Therefore $\Sigma$ has a maximal member $T$, since $M$ is a Noetherian $R$-module. Clearly, $T$ is not a strongly 2-absorbing submodule. Therefore, there exist $a, b, c \in K$ such that $abc(M/T)=0$ but $ab(M/T)\not=0$,  $ac(M/T)\not=0$, and  $bc(M/T)\not=0$. The maximality of $T$ implies that $M/ (T + abM)$, $M/ (T + acM)$,
and $M/ (T + bcM)$ have only finitely many minimal strongly 2-absorbing submodules.
Suppose $P/T$ is a minimal strongly 2-absorbing submodule of $M/T$. So
$abcM \subseteq T \subseteq P$, which implies that $abM \subseteq P$ or $acM \subseteq P$ or $bcM \subseteq P$. Thus $P/ (T + abM)$ is a minimal strongly 2-absorbing submodule of $M/ (T + abM)$ or $P/ (T + bcM)$ is a minimal strongly 2-absorbing submodule of $M/ (T + bcM)$ or $P/ (T + acM)$ is a minimal strongly 2-absorbing submodule of $M/ (T + acM)$.
Thus, there are only a finite number of possibilities for the submodule $M/T$. This is a
contradiction.
\end{proof}
\textbf{Acknowledgments.} The author would like to thank Professor Andrew Hetzel  for his helpful suggestions and useful comments.
\bibliographystyle{amsplain}

\end{document}